\documentclass[12pt]{elsarticle}
\usepackage{graphicx}
\graphicspath{ {Figures/} }
\usepackage{color}
\usepackage[rightcaption]{sidecap}
\usepackage{lineno}
\usepackage{multirow}
\usepackage{epstopdf}
\usepackage{rotating}
\usepackage{caption}
\usepackage{cite}
\usepackage{subcaption}
\biboptions{sort&compress}
\usepackage{amsmath,amssymb,amsfonts,amsthm}
\theoremstyle{plain}
\newtheorem{theorem}{Theorem}[section]
\newtheorem{corollary}{Corollary}[section]

\newtheorem{proposition}{Proposition}[section]
\newtheorem{definition}{Definition}[section]
\newtheorem{example}{Example}[section]

\theoremstyle{remark}

\usepackage[a4paper, total={6.8in, 9.2in}]{geometry}
\usepackage{hyperref}
\usepackage{tikz}
\journal{}
\makeatletter
\def\ps@pprintTitle{%
 \let\@oddhead\@empty
 \let\@evenhead\@empty
 \def\@oddfoot{\centerline{\thepage}}%
 \let\@evenfoot\@oddfoot}
\makeatother

\begin{document}
\begin{frontmatter}
\title{Some Difference Graphs}

\author[add1]{M. A. Seoud}
\ead{m.a.seoud@hotmail.com}
\author[add1,add2]{M. M. Farid\corref{cor1}}
\ead{mahmoud.faried@bue.edu.eg}
\author[add1]{M. Anwar}
\ead{mohammed.anwer@hotmail.com}

\cortext[cor1]{Department of Mathematics, Faculty of Science, Ain Shams University, Cairo, Egypt }
\address[add1]{Department of Mathematics, Faculty of Science, Ain Shams University, Cairo, Egypt }
\address[add2]{Basic Sciences Department, Faculty of Engineering, The British University in Egypt, Cairo, Egypt}
\begin{abstract}
In this paper, we discuss difference labeling of some standard families of graphs. We prove that Star, Butterfly, Bistar, umbrella and Olive tree are difference graphs. We also introduce  difference labelings for some snakes (double triangular snake, irregular triangular snake, alternate $C_n$ snake). Furthermore we introduce a corollary helps us to find a unique difference labeling for the complete graph $K_3$ and all forms of difference labeling for the Star graph. Also this corollary can be used to prove that the complete bipartite graph $K_{2,4}$ is not a difference graph but the proof is very lengthy.
\end{abstract}
\end{frontmatter}

\section{Introduction}

The notion of the autograph was introduced by \ G. S. Bloom, P. Hell, and H. Taylor \citep{1}. Harary \citep{2} called the autograph a difference graph. S. Bloom, Hell, and Taylor \citep{1} have shown that the following graphs are difference graphs: trees, $C_{n};K_{n};K_{n,n};K_{n,n-1},$ pyramids, and $n$-prisms. Gervacio \citep{3} proved that wheels $W_{n}$ are difference graphs if and only if $n=3;4,$ or $6$. Sonntag \citep{4} proved that cacti with girth at least $6$ are difference graphs, and he conjectured that all cacti are difference graphs. Sugeng and Ryan \citep{5} have provided difference labelings for cycles; fans; cycles with chords; graphs obtained by the one-point union of $K_{n}$ and $P_{m}$; and
graphs made from any number of copies of a given graph $G$ that has a difference labeling by identifying one vertex of the first with a vertex of the second, a different vertex of the second with the third and so on, In \citep{6}, Seoud and Helmi provided a survey of all graphs of order at most 5 and showed that the following graphs are difference graphs: $K_{n}$ for $n\geq 4$ with two deleted edges having no vertex in common; $K_{n}$ for $n\geq 6$ with three deleted edges having no vertex in common; gear graphs $G_{n}$ for $n \geq 3$; $P_m \times P_n$ for $m,n \geq 2$; triangular snakes; $C_4$-snakes; dragons; graphs consisting of two cycles of the same order joined by an edge, and graphs obtained by identifying the center of a star with a vertex of a cycle. The paper is organized as follows, the next section is devoted to some basic concepts. Some difference labellings for some graphs are introduced in Section 3.

In this paper we are following the basic definitions and notations for graph theory as in \citep{7,8,9,10}.

\section{Basic Concepts}
In this section we will recall some basic definitions and an important proposition.
\begin{definition}
A graph $G(V,E)$ is called a difference graph if there is a bijective map $f$ from $V$ to a set of positive integers $S$ such that $xy\in E$ if and only if $\left\vert f(x)-f(y)\right\vert \in S$, and $S$ is said to be the signature of $G$.
\end{definition}
\begin{definition}
A triangular snake $T_{n}$ is obtained from a path $P_{n}$ by replacing each edge of the $P_{n}$ by a cycle $C_{3}.$
\end{definition}

\begin{definition}
An alternate triangular snake $A(T_{n})$ is obtained from a path $P_{n}$ by replacing each alternate edge of $P_{n}$ by a cycle $C_{3}.$
\end{definition}

\begin{definition}
A double triangular snake $DT_{n}$ consists of two triangular snakes that have a common path.
\end{definition}
\begin{definition}
A $C_{n}$ snake is the graph obtained from a path $P_{n}$ by replacing each edge of the $P_{n}$ by a cycle $C_{n}.$
\end{definition}
\begin{definition}
An alternate $C_{n}$ snake $A(C_{n}\ snake)$ is the graph obtained from a path $P_{n}$ by replacing every alternate edge of the $P_{n}$ by a cycle $C_{n}$.
\end{definition}
\begin{definition}
An Olive tree $(T_{k})$ is a rooted tree consisting of $k$ branches where the $i^{th}$ branch is a path of length $i$.
\end{definition}
\begin{proposition}
\citep{2}
\end{proposition}

\begin{enumerate}
\item Vertex label values $s$ and $2s$ belong to adjacent vertices (first type),
\item Vertex label values $r$ and $t$ belong to vertices adjacent to a vertex labeled $r+t$ (second type),
\item No other adjacency occur in difference graphs.

\begin{proof}
(1) and (2) are obvious. To prove (3), note that if the vertices with labels $r$ and $r+t$ are adjacent, then$\left\vert (r+t)-r\right\vert =t$ belongs to $S$. Hence, either $t=r$ and we have an edge of the first kind, or $t\neq r $ and $\left\vert (r+t)-r\right\vert =t\in S$, i.e., we have an edge of the second kind.
\end{proof}
\end{enumerate}

\section{New Results}

\begin{corollary}
Let $S$ be the signature of a difference graph $G$,

\item[i)] For all $s\in S$,  $s \in \{2a,\frac{a}{2},a+b,|a-b|\}$ for some $a,b \in S $.
\item[ii)] The minimum label must be $\frac{a}{2}$ or $|a-b|$ for some $a$ and $b \in S $.
\item[iii)] The maximum label must be $2a$ or $a+b$ for some $ a,b \in S$.
\item[iv)] The degree of the vertex with the maximum label $s$ is odd if and only if it is adjacent to a vertex labeled by $\frac{s}{2}$.

\end{corollary}
\begin{proof}\begin{itemize}
\item [i)] It is clear using proposition $2.1$.
\item [ii)] Let $c$ be the minimum label in $G$ and let the vertex labeled by $c$ be  adjacent to  the vertex labeled by $a$ , hence $a-c\in S$, therefore, either $a-c=c$ i.e. $c=\frac{a}{2}$, or $a-c=b$ for some $b\in S$, which implies that $c=a-b.$
\item[iii)] Let $c$ be the maximum label in the $G$ and let the vertex labeled by $c$ be  adjacent to  the vertex labeled by $a$, hence $c-a\in S$, therefore, either $c-a=a$ i.e. $c=2a$, or $c-a=b $ for some $b\in S$, which implies that $c=a+b.$
\item[iv)] We have two types of labelings in the Difference graph. In case of the first type, the vertex with maximum label is adjacent to a vertex labeled by $\frac{s}{2},$ therefore it shares $1$ in the degree of the vertex with maximum label. So, the statement is done. In case of the existence of second type, since the vertex with maximum label is adjacent to two vertices such that the sum of their labels is $s$, it shares multiples of $2$ in the degree of the vertex with maximum label.
\end{itemize}
\end{proof}


\begin{theorem}
The complete graph $K_{3}$ is a difference graph with a unique signature form $S=\left\{ 3a,2a,a\right\} .$
\end{theorem}

\begin{proof}
Let $u_{1}$ be the vertex with the maximum label $3a$, then from the corollary $3.1,$ the label of the vertices $u_{2}$ and $u_{3}$ must be $\left\{ b,3a-b\right\} $. Without loss of generality one can assume that $3a-b>b$ , therefore $3a-b-b=3a-2b\in S.$ Consequently, $3a-2b=b,$  which gives $a=b.$ Then, the labels of $u_{2}$ and $u_{3}$ will be $\left\{ a,2a\right\} .$
\end{proof}

\begin{theorem}
The star graph is a difference graph. Moreover, $S$ is a signature of the star graph if and only if it has the one of the forms:
\end{theorem}

\begin{enumerate}
\item[i)] If $n$ is even, then $S=\left\{ a,b_{i},a-b_{i}\right\}$, such that $a$ is the maximum label and  $\left\vert b_{i}-(a-b_{i})\right\vert \notin S$, or  $S=\{4a,2a,a,b_{i},2a-b_{i}\}$ such that the difference between any two elements of \\ $\left\{ 4a,a,b_{i},2a-b_{i}\right\} \notin S.$
\item[ii)] If $n$ is odd, then $S=$\{$2a,a,b_{i},2a-b_{i}\}$ such that the difference between any two elements of $\left\{ a,b_{i},2a-b_{i}\notin S\right\}$ or  $S=\{2a,a,b_{i},a-b_{i}\}$ such that the difference between any two elements of $\left\{ 2a,b_{i},a-b_{i}\right\} \notin S.$
\end{enumerate}

\begin{proof}
i) Let $S_{n}$ be a star graph where $n$ is even integer, let $u_{0}$ be the vertex of degree $n$ with a maximum label $a$ , then from the corollary $3.1$ the label of $\left\{ u_{1},u_{2},\cdots ,u_{n}\right\} $ must be $\left\{ b_{i},a-b_{i}\right\} $ such that $\left\vert b_{i}-(a-b_{i})\right\vert \neq b_{j}$ or $a-b_{j}.$ On the other hand, let $u_{1}$ be any vertex of degree $1$ with a maximum label $4a$ , then from the corollary $3.1,$ and where the degree of $u_{1}$ is one, the label of $u_{\circ }$ must be $2a,$\ since no vertex from $u_{2},u_{3},\cdots ,u_{n}$ can be labeled bigger than the label of $u_{\circ }$ (if that happened its label must be $4a $ which is rejected).Then $u_{\circ }$\ will be the vertex with maximum label of the vertices $\left\{ u_{\circ },u_{2},u_{3},\cdots ,u_{n}\right\} $, since the number of elements of the set $\left\{ u_{2},u_{3},\cdots ,u_{n}\right\} $ is odd. Then from corollary $3.1$ the labels of the vertices $\left\{ u_{2},u_{3},\cdots ,u_{n}\right\} $ will be \{$a,b_{i},2a-b_{i}$ ; such that the differences between any two elements of $\left\{ 4a,a,2a-b_{i}\notin S\right\} .$
ii) By the same way.
\end{proof}




\begin{example}
 A difference labeling of the star graph $S_{8}$ is illustrated in Fig.\ref{figure3}
\end{example}

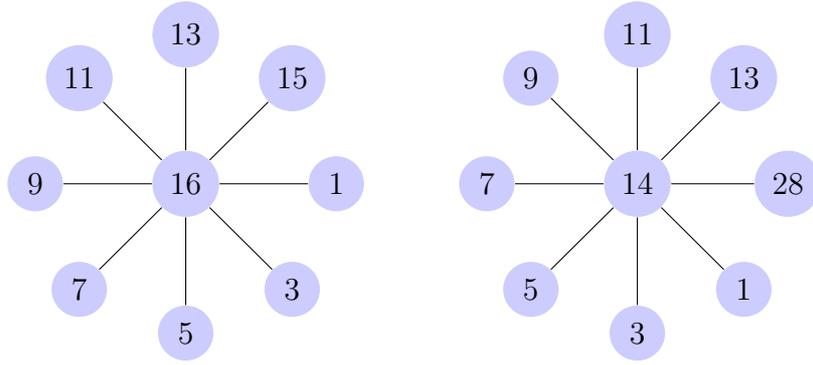
\begin{figure}[h!]
\centering
\begin{tikzpicture}
[scale=.7,auto=center,every node/.style={circle,fill=blue!20}]
  \node (u0) at (0,0) {16};
 \node (u1) at ({sqrt(8)},0) {1};
  \node (u2) at (2,2) {15};
  \node (u3) at (0,{sqrt(8)}){13};
 \node (u4) at (-2,2)  {11};
  \node (u5) at ({-sqrt(8)},0) {9};
\node (u6) at (-2,-2) {7};
  \node (u7) at (0,{-sqrt(8)}) {5};
  \node (u8) at (2,-2) {3};

 \draw (u0) -- (u1);
  \draw (u0) -- (u2);
  \draw (u0) -- (u3);
  \draw (u0) -- (u4);
 \draw (u0) -- (u5);
\draw (u0) -- (u6);
  \draw (u0) -- (u7);
\draw (u0) -- (u8);

  \node (v0) at ({3*sqrt(8)},0) {14};
 \node (v1) at ({4*sqrt(8)},0) {28};
  \node (v2) at ({2+3*sqrt(8)},2) {13};
  \node (v3) at ({3*sqrt(8)},{sqrt(8)}){11};
 \node (v4) at ({-2+3*sqrt(8)},2)  {9};
  \node (v5) at ({2*sqrt(8)},0) {7};
\node (v6) at ({-2+3*sqrt(8)},-2) {5};
  \node (v7) at ({3*sqrt(8)},{-sqrt(8)}) {3};
  \node (v8) at ({2+3*sqrt(8)},-2) {1};

 \draw (v0) -- (v1);
  \draw (v0) -- (v2);
  \draw (v0) -- (v3);
  \draw (v0) -- (v4);
 \draw (v0) -- (v5);
\draw (v0) -- (v6);
  \draw (v0) -- (v7);
\draw (v0) -- (v8);
\end{tikzpicture}
  \caption{Star graph $S_8$ labeling}\label{figure3}
\end{figure}
\begin{theorem}
The Butterfly graph is a difference graph.
\end{theorem}

\begin{proof}
Let the Butterfly graph be described as indicated in Fig.\ref{figure4}.

\begin{figure*}[h!]
    \centering
        \begin{tikzpicture}
  [scale=.9,auto=center,every node/.style={circle,fill=blue!20}]

\node(w0) at (0,0){$w_0$};
\node(u0) at (-4,2){$w_1$};
\node(u1) at (-2,2){$u_0$};
\node(u2) at (0,2){$u_1$};
\node(u3) at (2,2){$u_2$};
\node(un) at (4,2){$u_n$};

\node(v0) at (-4,-2){$v_0$};
\node(v1) at (-2,-2){$v_1$};
\node(v2) at (0,-2){$v_2$};
\node(v3) at (2,-2){$v_3$};
\node(vm) at (4,-2){$v_m$};

\draw(w0) -- (u0);
\draw(w0) -- (u1);
\draw(w0) -- (u2);
\draw(w0) -- (u3);
\draw(w0) -- (un);

\draw(w0) -- (v1);
\draw(w0) -- (v2);
\draw(w0) -- (v3);
\draw(w0) -- (vm);
\draw(w0) -- (v0);

\draw(u1) -- (u0);
\draw(u1) -- (u2);
\draw(u2) -- (u3);
\draw  [thick, dotted , -] (u3) -- (un);
\draw(u0) -- (u1);
\draw(v1) -- (v2);
\draw(v2) -- (v3);
\draw(v3) -- (vm);
\draw(v1) -- (v0);
\end{tikzpicture}
\caption{Butterfly graph.}
  \label{figure4}
\end{figure*}
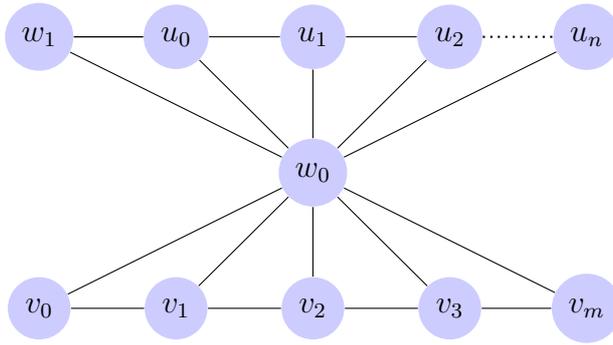
Define a labeling function $f:V($Butterfly graph$)\longrightarrow S$ (positive integers) as follows:
\begin{eqnarray*}
f\left( w_{0}\right) &=& 6.\\
f\left( w_{1}\right) &=& 2.\\
f\left( u_{i}\right) &=& 4+6i ,\ \ i=0,1,2,\ldots .\\
f\left( v_{j}\right) &=& 3+6j ,\  \ j=0,1,2,\ldots .
\end{eqnarray*}
\end{proof}

\begin{example}
A difference labeling of the butterfly graph is illustrated in  Fig.\ref{fig51}.
\end{example}

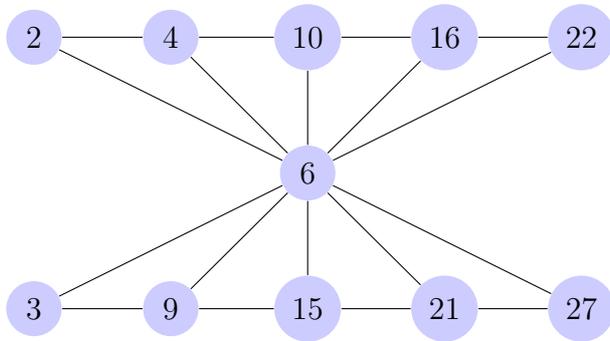
\begin{figure*}[h!]
        \centering
        \begin{tikzpicture}
  [scale=.9,auto=center,every node/.style={circle,fill=blue!20}]

\node(w0) at (0,0){6};
\node(u1) at (-4,2){2};
\node(u2) at (-2,2){4};
\node(u3) at (0,2){10};
\node(u4) at (2,2){16};
\node(u5) at (4,2){22};

\node(v1) at (-4,-2){3};
\node(v2) at (-2,-2){9};
\node(v3) at (0,-2){15};
\node(v4) at (2,-2){21};
\node(v5) at (4,-2){27};

\draw(w0) -- (u1);
\draw(w0) -- (u2);
\draw(w0) -- (u3);
\draw(w0) -- (u4);
\draw(w0) -- (u5);

\draw(w0) -- (v1);
\draw(w0) -- (v2);
\draw(w0) -- (v3);
\draw(w0) -- (v4);
\draw(w0) -- (v5);

\draw(u1) -- (u2);
\draw(u2) -- (u3);
\draw(u3) -- (u4);
\draw(u4) -- (u5);
\draw(v1) -- (v2);
\draw(v2) -- (v3);
\draw(v3) -- (v4);
\draw(v4) -- (v5);
\end{tikzpicture}
\caption{butterfly graph labeling.}
\label{fig51}
\end{figure*}

\begin{theorem}
The Bistar graph is a difference graph.
\end{theorem}
\begin{proof}
Let the Bistar graph be described as indicated in Fig.\ref{figure55}.
\begin{figure*}[h!]

    \centering

        \begin{tikzpicture}
  [scale=.9,auto=center,every node/.style={circle,fill=blue!20}]
\node(u0) at (0,0){$u_0$};
\node(u1) at (-4,2){$u_1$};
\node(u2) at (-2,2){$u_2$};
\node(u3) at (0,2){$u_3$};
\node(u4) at (2,2){$u_4$};
\node(u5) at (4,2){$u_n$};

\node(v0) at (0,-2){$v_0$};
\node(v1) at (-4,-4){$v_1$};
\node(v2) at (-2,-4){$v_2$};
\node(v3) at (0,-4){$v_3$};
\node(v4) at (2,-4){$v_4$};
\node(v5) at (4,-4){$v_m$};

\draw(u0) -- (u1);
\draw(u0) -- (u2);
\draw(u0) -- (u3);
\draw(u0) -- (u4);
\draw(u0) -- (u5);

\draw(v0) -- (v1);
\draw(v0) -- (v2);
\draw(v0) -- (v3);
\draw(v0) -- (v4);
\draw(v0) -- (v5);

\draw(u0) -- (v0);
\end{tikzpicture}
\caption{Bistar graph.}
    \label{figure55}
\end{figure*}
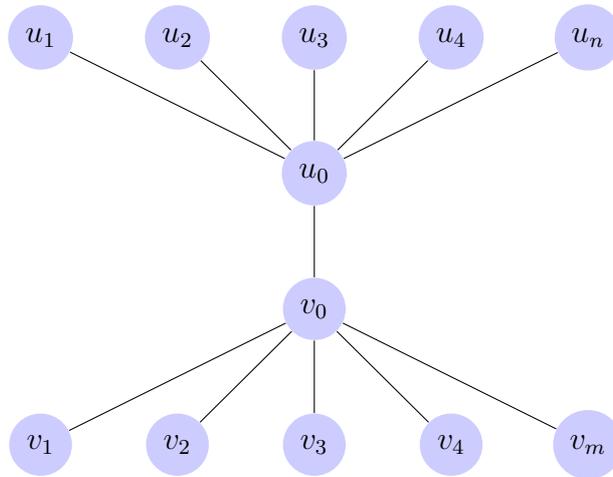
\\
 Define the labeling function $f:V($Bistar graph$)\longrightarrow S$ (positive integers) as follows:
\begin{eqnarray*}
f\left( u_{0}\right) &=& 2n.\\
f\left( u_{i}\right) &=& 2i-1, \ \ i=1,2,\ldots ,n.\\
f\left( v_{j}\right) &=& 2n+j\left( 2n+2\right),\ \  j=1,2,\ldots,m.\\
f\left( v_{0}\right) &=& f\left( u_{\circ }\right) +f\left( v_{m}\right)= 4n+2m\left( n+1\right).\\
\end{eqnarray*}
\end{proof}
\begin{example}
A difference labeling of the Bistar graph is illustrated in  Fig.\ref{figure555}.
\end{example}
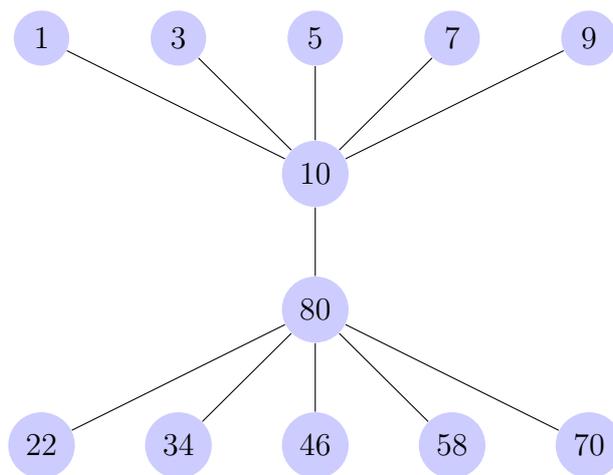
\begin{figure*}[h!]

        \centering

        \begin{tikzpicture}
  [scale=.9,auto=center,every node/.style={circle,fill=blue!20}]
\node(u0) at (0,0){10};
\node(u1) at (-4,2){1};
\node(u2) at (-2,2){3};
\node(u3) at (0,2){5};
\node(u4) at (2,2){7};
\node(u5) at (4,2){9};

\node(v0) at (0,-2){$80$};
\node(v1) at (-4,-4){$22$};
\node(v2) at (-2,-4){$34$};
\node(v3) at (0,-4){$46$};
\node(v4) at (2,-4){$58$};
\node(v5) at (4,-4){$70$};

\draw(u0) -- (u1);
\draw(u0) -- (u2);
\draw(u0) -- (u3);
\draw(u0) -- (u4);
\draw(u0) -- (u5);

\draw(v0) -- (v1);
\draw(v0) -- (v2);
\draw(v0) -- (v3);
\draw(v0) -- (v4);
\draw(v0) -- (v5);
\draw(u0) -- (v0);
\end{tikzpicture}
\caption{Bistar graph labeling.}
\label{figure555}
\end{figure*}

\newpage

\begin{theorem}
The Umbrella graph is a difference graph.
\end{theorem}

\begin{proof}
Let the Umbrella graph be described as indicated in Figure.\ref{figure65}. We define the labeling function $f:V($Umbrella graph$)\longrightarrow S$ (positive integers) as follows:

\begin{eqnarray*}
f\left( u_{0}\right) &=& 2.\\
f\left( u_{i}\right)  &=& 2i-1,\ \  i=1,2,\ldots ,n.\\
f\left( v_{0}\right) &=& f\left( u_{\circ }\right) +f\left( v_{1}\right)=4n+2 .\\
f\left( v_{j}\right) &=& 2^{j-1} \cdot 4n, \ j=0,1,2,\ldots .
\end{eqnarray*}
\end{proof}
\begin{figure*}[h!]

        \centering

        \begin{tikzpicture}
  [scale=.6,auto=center,every node/.style={circle,fill=blue!20}]

\node(u0) at (0,0){$u_0$};
\node(u1) at (-4,2){$u_1$};
\node(u2) at (-2,2){$u_2$};
\node(u3) at (0,2){$u_3$};
\node(u4) at (2,2){$u_4$};
\node(u5) at (4,2){$u_n$};

\node(v0) at (0,-1.5){$v_0$};
\node(v1) at (0,-3){$v_1$};
\node(v4) at (0,-4.5){$v_2$};
\node(v5) at (0,-6){$v_m$};

\draw(u0) -- (u1);
\draw(u0) -- (u2);
\draw(u0) -- (u3);
\draw(u0) -- (u4);
\draw(u0) -- (u5);

\draw(v0) -- (v1);
\draw(u0) -- (v0);
\draw(v1) -- (v4);
\draw [thick, dotted , -] (v4) -- (v5);

\draw(u1) -- (u2);
\draw(u2) -- (u3);
\draw(u3) -- (u4);
\draw [thick, dotted , -] (u4) -- (u5);
\end{tikzpicture}
\caption{Umbrella graph}
\label{figure65}
\end{figure*}
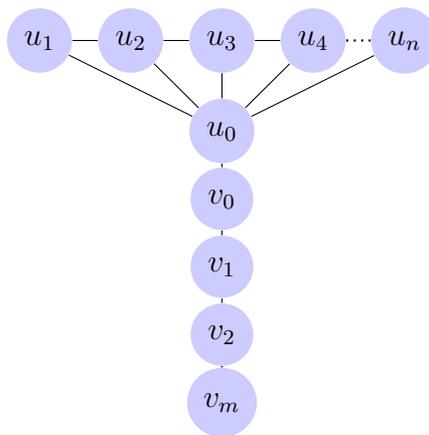
\begin{example}
A difference labeling of the Umbrella garph is illustrated in  Fig.\ref{figure66}
\end{example}

\begin{figure*}[h!]
    \centering
        \begin{tikzpicture}
  [scale=.9,auto=center,every node/.style={circle,fill=blue!20}]

\node(u0) at (0,0){2};
\node(u1) at (-4,2){1};
\node(u2) at (-2,2){3};
\node(u3) at (0,2){5};
\node(u4) at (2,2){7};
\node(u5) at (4,2){9};

\node(v0) at (0,-1){22};
\node(v1) at (0,-2){20};
\node(v4) at (0,-3){40};
\node(v5) at (0,-4){80};

\draw(u0) -- (u1);
\draw(u0) -- (u2);
\draw(u0) -- (u3);
\draw(u0) -- (u4);
\draw(u0) -- (u5);

\draw(v0) -- (v1);
\draw(u0) -- (v0);
\draw(v1) -- (v4);
\draw(v4) -- (v5);

\draw(u1) -- (u2);
\draw(u2) -- (u3);
\draw(u3) -- (u4);
\draw(u4) -- (u5);
\end{tikzpicture}
\caption{Umbrella graph labeling.}
\label{figure66}
\end{figure*}
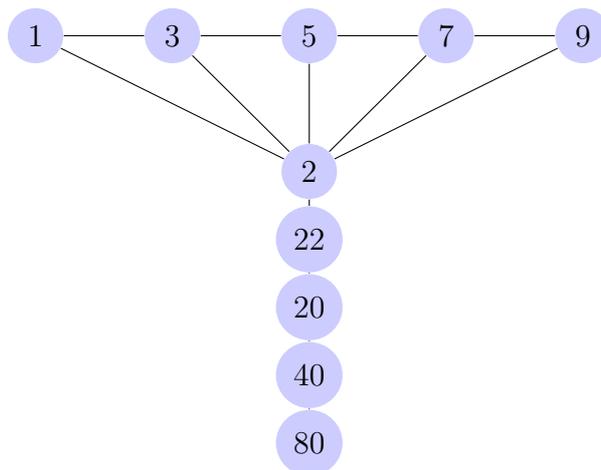

\newpage

\begin{theorem}
Double triangular snake $DT_{n}$ is a difference graph.
\end{theorem}

\begin{proof}
Let the double triangular snake $DT_{n}$ be described as indicated in Fig.\ref{figure10}.
\begin{figure}[h!]
  \centering
  \begin{tikzpicture}
  [scale=.9,auto=center,every node/.style={circle,fill=blue!20}]

\node(u1) at ( -5,0){$u_1$};
\node(u2) at (-3,0){$u_2$};
\node(u4) at (1,0){$u_4$};
\node(u3) at (-1,0){$u_3$};
\node(un) at (3,0){$u_n$};
\node(un+1) at (5,0){$u_{n+1}$};

\node(v1) at (-4,2){$v_1$};
\node(v2) at (-2,2){$v_2$};
\node(v3) at (0,2){$v_3$};
\node(vn) at (4,2){$v_n$};

\node(w1) at (-4,-2){$w_1$};
\node(w2) at (-2,-2){$w_2$};
\node(w3) at (0,-2){$w_3$};
\node(wn) at (4,-2){$w_n$};

\draw(u1) -- (u2);
\draw(u2) -- (u3);
\draw(u3) -- (u4);
\draw  [thick, dotted , -] (u4) -- (un);
\draw(un) -- (un+1);

\draw(u1) -- (v1);
\draw(u2) -- (v1);
\draw(u2) -- (v2);
\draw(u3) -- (v2);
\draw(u3) -- (v3);
\draw(u4) -- (v3);
\draw(un) -- (vn);
\draw(un+1) -- (vn);

\draw(u1) -- (w1);
\draw(u2) -- (w1);
\draw(u2) -- (w2);
\draw(u3) -- (w2);
\draw(u3) -- (w3);
\draw(u4) -- (w3);
\draw(un) -- (wn);
\draw(un+1) -- (wn);

\end{tikzpicture}
  \caption{The double triangular snake $DT_{n}$.}\label{figure10}
\end{figure}
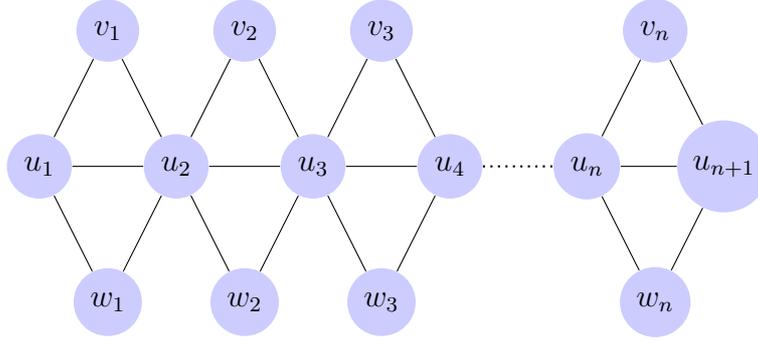\\
Define the labeling function $f:V(DT_{n}\ snake))\longrightarrow S$ (positive integers) as follows
\begin{eqnarray*}
f(u_{i}) &=& 3^{i-1}\cdot2^{n-(i-1)} \ ,i=1,2, \dots , n+1.\\
f(v_{i})&=& 5\cdot 3^{i-1}\cdot2^{n-i} \ ,i=1,2, \dots , n.\\
f(w_{i}) &=& 3^{i-1}\cdot2^{n-i}\ ,i=1,2, \dots , n.
\end{eqnarray*}
\end{proof}

\begin{example}
A difference labeling of the double triangular snake is illustrated in Fig.\ref{figure11}
\end{example}

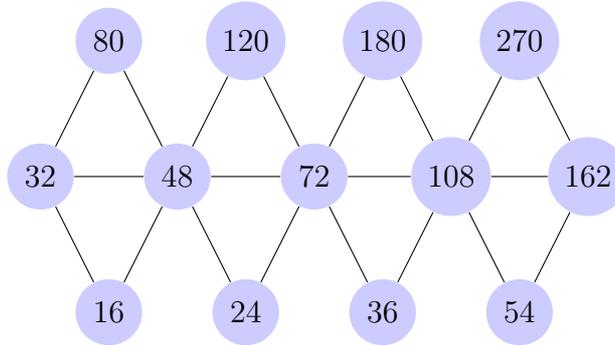
\begin{figure}[h!]

  \centering

  \begin{tikzpicture}
  [scale=.9,auto=center,every node/.style={circle,fill=blue!20}]

\node(u1) at ( -5,0){32};
\node(u2) at (-3,0){48};
\node(u4) at (1,0){108};
\node(u3) at (-1,0){72};
\node(un) at (3,0){162};

\node(v1) at (-4,2){80};
\node(v2) at (-2,2){120};
\node(v3) at (0,2){180};
\node(vn) at (2,2){270};

\node(w1) at (-4,-2){16};
\node(w2) at (-2,-2){24};
\node(w3) at (0,-2){36};
\node(wn) at (2,-2){54};

\draw(u1) -- (u2);
\draw(u2) -- (u3);
\draw(u3) -- (u4);
\draw(u4) -- (un);

\draw(u1) -- (v1);
\draw(u2) -- (v1);
\draw(u2) -- (v2);
\draw(u3) -- (v2);
\draw(u3) -- (v3);
\draw(u4) -- (v3);
\draw(un) -- (vn);
\draw(u4) -- (vn);

\draw(u1) -- (w1);
\draw(u2) -- (w1);
\draw(u2) -- (w2);
\draw(u3) -- (w2);
\draw(u3) -- (w3);
\draw(u4) -- (w3);
\draw(un) -- (wn);
\draw(u4) -- (wn);

\end{tikzpicture}
  \caption{The double triangular snake $DT_{5}$.}\label{figure11}
\end{figure}

\begin{theorem}
The irregular triangular snake $\left( IT_{n}\right) $ is a difference graph.
\end{theorem}

\begin{proof}
Assume that the irregular triangular snake $\left( IT_{n}\right) $ is described as indicated in Fig.\ref{figure12}.
Define the labeling function $f:V(A(C_{n}\ snake))\longrightarrow S$ (positive integers) as follows:

\begin{eqnarray*}
f(u_{i}) &=& 2^{i}  \ ,i=1,2, \ldots , n.\\
f(v_{j}) &=& f(u_{j})+f(u_{j+2})=5 \cdot 2^{j}  \ ,i=1,3,5, \ldots , n-3.\\
f(w_{k}) &=& f(v_{k})+f(v_{k+2})=5^{2} \cdot 2^{k},  \ ,i=1,2, \ldots , n-3.\\
f(w_{n-3}) &=& f(v_{n-3})+f(u_{n}).
\end{eqnarray*}
\end{proof}

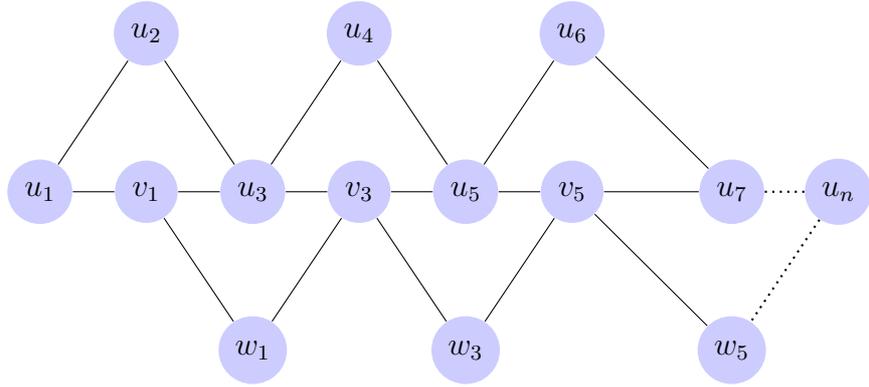
\begin{figure}[h!]

  \centering

  \begin{tikzpicture}
  [scale=.7,auto=center,every node/.style={circle,fill=blue!20}]

\node(u1) at ( -8,0){$u_1$};
\node(u2) at (-6,3){$u_2$};
\node(u4) at (-2,3){$u_4$};
\node(u3) at (-4,0){$u_3$};
\node(u5) at (0,0){$u_5$};
\node(u6) at (2,3){$u_6$};
\node(u7) at (5,0){$u_7$};
\node(un) at (7,0){$u_{n}$};

\node(v1) at (-6,0){$v_1$};
\node(v3) at (-2,0){$v_3$};
\node(v5) at (2,0){$v_5$};

\node(w1) at (-4,-3){$w_1$};
\node(w3) at (0,-3){$w_3$};
\node(w5) at (5,-3){$w_5$};

\draw(u1) -- (u2);
\draw(u2) -- (u3);
\draw(u3) -- (u4);
\draw(u4) -- (u5);
\draw(u5) -- (u6);
\draw  (u6) -- (u7);
\draw [thick, dotted , -] (u7) -- (un);

\draw(u1) -- (v1);
\draw(v1) -- (u3);
\draw(u3) -- (v3);
\draw(v3) -- (u5);
\draw(u5) -- (v5);
\draw (v5) -- (u7);

\draw(v1) -- (w1);
\draw(w1) -- (v3);
\draw(v3) -- (w3);
\draw(w3) -- (v5);
\draw(v5) -- (w5);
\draw [thick, dotted , -] (un) -- (w5);

\end{tikzpicture}
  \caption{The irregular triangular snake.}\label{figure12}
\end{figure}

\begin{example}
A difference labeling of the irregular triangular snake labeling is illustrated in Fig.\ref{figure13}
\end{example}

\begin{figure}[h!]

  \centering

  \begin{tikzpicture}
  [scale=.7,auto=center,every node/.style={circle,fill=blue!20}]

\node(u1) at ( -8,0){2};
\node(u2) at (-6,3){4};
\node(u3) at (-4,0){8};
\node(u4) at (-2,3){16};
\node(u5) at (0,0){32};
\node(u6) at (2,3){64};
\node(un) at (5,0){128};
\node(un+1) at (7,0){256};

\node(v1) at (-6,0){10};
\node(v3) at (-2,0){40};
\node(v5) at (2,0){160};

\node(w1) at (-4,-3){50};
\node(w3) at (0,-3){200};
\node(w5) at (5,-3){416};

\draw(u1) -- (u2);
\draw(u2) -- (u3);
\draw(u3) -- (u4);
\draw(u4) -- (u5);
\draw(u5) -- (u6);
\draw(u6) -- (un);
\draw(un) -- (un+1);

\draw(u1) -- (v1);
\draw(v1) -- (u3);
\draw(u3) -- (v3);
\draw(v3) -- (u5);
\draw(u5) -- (v5);
\draw(v5) -- (un);
\draw(un) -- (un+1);

\draw(v1) -- (w1);
\draw(w1) -- (v3);
\draw(v3) -- (w3);
\draw(w3) -- (v5);
\draw(v5) -- (w5);
\draw(un+1) -- (w5);

\end{tikzpicture}
  \caption{The irregular triangular snake $IT_{8}$ labeling}\label{figure13}
\end{figure}
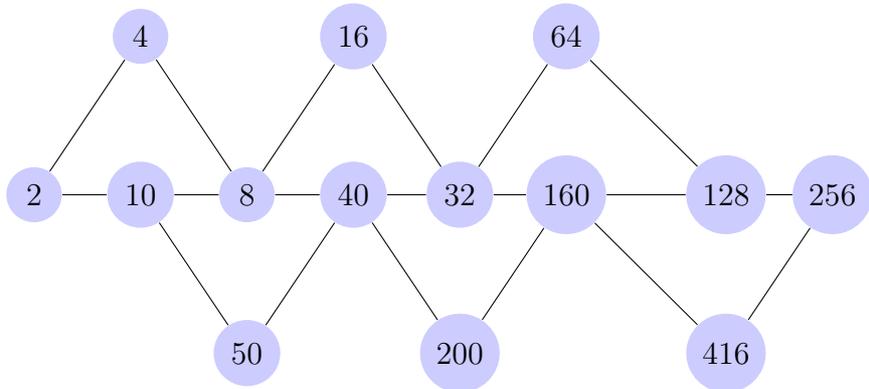

\begin{theorem}
The $C_{n}$ snake is difference graph.
\end{theorem}

\begin{proof}
Let the $C_{n}$ snake be described as indicated in Fig.\ref{figure16}

\begin{figure}[h!]
  \centering
\begin{tikzpicture}
  [scale=.9,auto=center,every node/.style={circle,fill=blue!20}]
  \node (u1) at (-15,0) {$u_0$};
 \node (u2) at (-14.5,1)  {$u_1$};
  \node (u3) at (-14,2)  {$u_2$};
  \node (u4) at (-13,2) {$u_3$};
 \node (u5) at (-12.5,1)  {$u_4$};
  \node (u6) at (-12,0)  {$u_n$};

\node (u7) at (-11.5,1)  {};
  \node (u8) at (-11,2)  {};
  \node (u9) at (-10,2) {};
 \node (u10) at (-9.5,1)  {};
  \node (u11) at (-9,0)  {};

\node (u12) at (-8.5,1)  {};
 \node (u13) at (-8,2)  {};
 \node (u14) at (-7,2) {};
 \node (u15) at (-6.5,1)  {};
 \node (u16) at (-6,0)  {};

\node (u17) at (-5.5,1)  {};
  \node (u18) at (-5,2)  {};
  \node (u19) at (-4,2) {};
 \node (u20) at (-3.5,1)  {};
  \node (u21) at (-3,0)  {};

  \draw (u1) -- (u2);
  \draw (u2) -- (u3);
  \draw (u3) -- (u4);
 \draw (u4) -- (u5);
  \draw [thick, dotted , -](u5) -- (u6);
  \draw (u1) -- (u6);

  \draw (u7) -- (u8);
  \draw (u8) -- (u9);
  \draw (u9) -- (u10);
 \draw (u10) -- (u11);
  \draw (u7) -- (u6);
  \draw (u11) -- (u6);

  \draw (u11) -- (u12);
  \draw (u12) -- (u13);
  \draw (u13) -- (u14);
 \draw (u14) -- (u15);
  \draw (u15) -- (u16);
  \draw (u11) -- (u16);

  \draw (u16) -- (u17);
  \draw (u17) -- (u18);
  \draw (u18) -- (u19);
  \draw (u19) -- (u20);
 \draw (u20) -- (u21);
  \draw (u16) -- (u21);
\end{tikzpicture}
  \caption{$C_n$ snake}\label{figure16}
\end{figure}
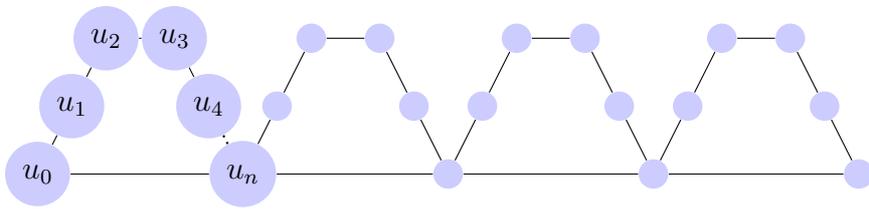
Define the labeling function $f:V(C_{n}\ snake)\longrightarrow S$ (positive integers) as follows:\\
$$f(u_{i})=2^{j}\ (1+2^{n-2})^{\left\lfloor \frac{i}{n-1}\right\rfloor },\qquad  i=0,1,2,\ldots,$$
where $j=rem(i,n-1)$ is the remainder of $i$ when it divided by $n-1$.
\end{proof}

\begin{example}
A difference labeling of the $C_{6}$ snake is illustrated in Fig.\ref{figure17}.
\end{example}

\begin{figure}[h!]
\centering
\begin{tikzpicture}
 [acteur/.style={circle, fill=blue!40,thick, inner sep=2pt, minimum size=0.3cm}]

  \node (u1) at (-15,0) [acteur][label=$2^0$]{};
 \node (u2) at (-14,1) [acteur][label=$2$]{};
  \node (u3) at (-13,2) [acteur][label=$2^2$]{};
  \node (u4) at (-12,2) [acteur][label=$2^3$]{};
 \node (u5) at (-11,1)  [acteur][label=$2^4$]{};
  \node (u6) at (-10,0) [acteur][label=$17$]{};

\node (u7) at (-9,1) [acteur][label=$2 \cdot 17$]{};
  \node (u8) at (-8,2) [acteur][label=$2^2 \cdot 17$]{};
  \node (u9) at (-7,2) [acteur][label=$2^3 \cdot 17$]{};
 \node (u10) at (-6,1) [acteur][label=$2^4 \cdot 17$]{};
  \node (u11) at (-5,0)  [acteur][label=$17^2$]{};

\node (u12) at (-4,1)  [acteur][label=2 $17^2$]{};
 \node (u13) at (-3,2)  [acteur][label=$2^2 \cdot 17^2$]{};
 \node (u14) at (-1.5,2)  [acteur][label=$2^3 \cdot 17^2$]{};
 \node (u15) at (-.5,1)   [acteur][label=$2^4 \cdot 17^2$]{};
 \node (u16) at (.5,0)   [acteur][label=$ 17^3$]{};

  \draw (u1) -- (u2);
  \draw (u2) -- (u3);
  \draw (u3) -- (u4);
 \draw (u4) -- (u5);
  \draw (u5) -- (u6);
  \draw (u1) -- (u6);

  \draw (u7) -- (u8);
  \draw (u8) -- (u9);
  \draw (u9) -- (u10);
 \draw (u10) -- (u11);
  \draw (u7) -- (u6);
  \draw (u11) -- (u6);

  \draw (u11) -- (u12);
  \draw (u12) -- (u13);
  \draw (u13) -- (u14);
 \draw (u14) -- (u15);
  \draw (u15) -- (u16);
  \draw (u11) -- (u16);
\end{tikzpicture}\caption{$C_6$ snake labeling}\label{figure17}
\end{figure}
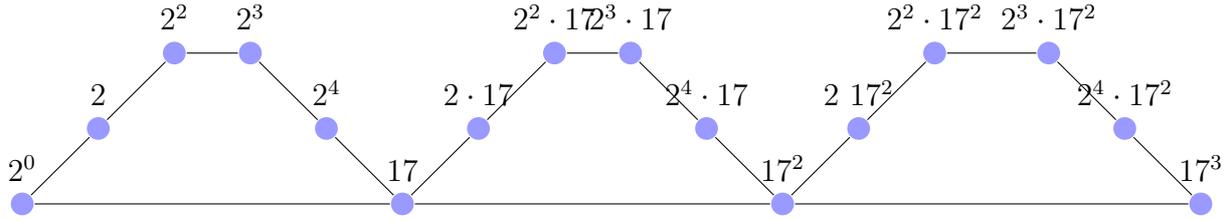

\begin{theorem}
The alternate $C_{n}$ snake $A(C_{n}\ snake)$ is a difference graph.
\end{theorem}
\begin{proof}
Let the alternate $C_{n}$ snake be described as indicated in Fig.\ref{figure18}.
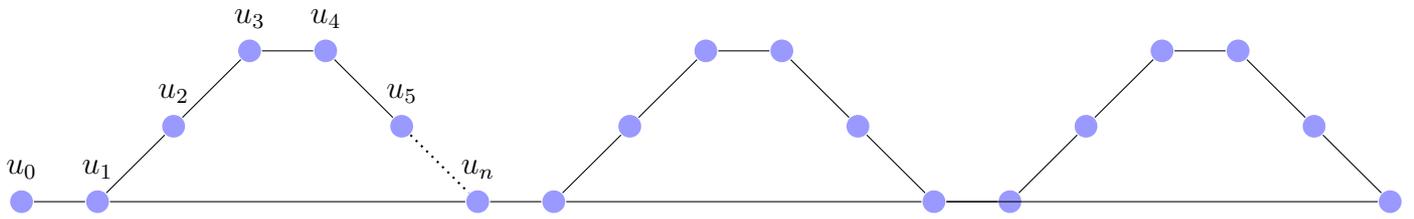
\begin{figure}[h!]
  \centering
  \begin{tikzpicture}
 [acteur/.style={circle, fill=blue!40,thick, inner sep=2pt, minimum size=0.3cm}]

\node (u0) at (-16,0) [acteur][label=$u_0$]{};
  \node (u1) at (-15,0) [acteur][label=$u_1$]{};
 \node (u2) at (-14,1) [acteur][label=$u_2$]{};
  \node (u3) at (-13,2) [acteur][label=$u_3$]{};
  \node (u4) at (-12,2) [acteur][label=$u_4$]{};
 \node (u5) at (-11,1)  [acteur][label=$u_5$]{};
  \node (u6) at (-10,0) [acteur][label=$u_n$]{};

\node (u7) at (-9,0) [acteur][label=$$]{};
  \node (u8) at (-8,1) [acteur][label=$$ ]{};
  \node (u9) at (-7,2) [acteur][label=$$ ]{};
 \node (u10) at (-6,2) [acteur][label=$$ ]{};
  \node (u11) at (-5,1)  [acteur][label=$$]{};
\node (u12) at (-4,0)  [acteur][label= $$]{};

 \node (u13) at (-3,0)  [acteur][label=$$]{};
 \node (u14) at (-2,1)  [acteur][label=$$]{};
 \node (u15) at (-1,2)   [acteur][label=$$]{};
 \node (u16) at (0,2)   [acteur][label=$$]{};
 \node (u17) at (1,1)  [acteur][label=$$]{};
 \node (u18) at (2,0)  [acteur][label=$$]{};

 \draw (u0) -- (u1);
  \draw (u1) -- (u2);
  \draw (u2) -- (u3);
  \draw (u3) -- (u4);
 \draw (u4) -- (u5);
 \draw [thick, dotted , -] (u5) -- (u6);
\draw (u1) -- (u6);

  \draw (u7) -- (u12);
\draw (u6) -- (u7);
  \draw (u7) -- (u8);
  \draw (u8) -- (u9);
  \draw (u9) -- (u10);
 \draw (u10) -- (u11);
  \draw (u11) -- (u12);

  \draw (u12) -- (u18);
  \draw (u12) -- (u13);
  \draw (u13) -- (u14);
 \draw (u14) -- (u15);
  \draw (u15) -- (u16);
  \draw (u17) -- (u16);
  \draw (u17) -- (u18);

\end{tikzpicture}
  \caption{Alternate $C_n$ snake.}\label{figure18}
\end{figure}

Define the labeling function $f:V(A(C_{n}\ snake))\longrightarrow S$ (positive integers) as follows:
$$f(u_{i})=2^{j+\left\lfloor \frac{i}{n}\right\rfloor }\
(1+2^{n-2})^{\left\lfloor \frac{i}{n}\right\rfloor }, \quad i=0,1,2, \ldots$$
where $j=rem(i,n)$ is the remainder of $i$ when it divided by $n$.
\end{proof}
\begin{example}
A difference labeling of the alternate $C_{5}$ snake $A(C_{5}\ snake)$ is illustrated in Fig.\ref{figure99}
\end{example}
\begin{figure}[h!]
    \centering
\begin{tikzpicture}
 [acteur/.style={circle, fill=blue!40,thick, inner sep=2pt, minimum size=0.3cm}]
    \node (u20) at (-16,0) [acteur][label=$2^0$]{};
  \node (u0) at (-15,0) [acteur][label=$2^1$]{};
 \node (u1) at (-14.5,1) [acteur][label=$2^2$]{};
  \node (u2) at (-13.5,2) [acteur][label=$2^3$]{};
  \node (u3) at (-12.5,1) [acteur][label=$2^4$]{};
 \node (u4) at (-12,0)  [acteur][label=$2^1 \cdot 9$]{};

 \node (u5) at (-11,0) [acteur][label=$2^2 \cdot 9$]{};
\node (u6) at (-10.5,1) [acteur][label=$2^3 \cdot 9$]{};
  \node (u7) at (-9.5,2) [acteur][label=$2^4 \cdot 9$ ]{};
  \node (u8) at (-8.5,1) [acteur][label=$2^5 \cdot 9$ ]{};
 \node (u9) at (-8,0) [acteur][label=$2^2 \cdot 9^2$ ]{};

\node (u10) at (-7,0)  [acteur][label=$2^3 \cdot 9^2$]{};
\node (u11) at (-6.5,1)  [acteur][label= $2^4 \cdot 9^2$]{};
 \node (u12) at (-5.5,2)  [acteur][label=$2^5\cdot9^2$]{};
 \node (u13) at (-4.5,1)  [acteur][label=$2^6 \cdot 9^2$]{};
 \node (u14) at (-4,0)   [acteur][label=$2^3 \cdot 9^3$]{};

\node (u15) at (-3,0)  [acteur][label=$2^4 \cdot 9^3$]{};
\node (u16) at (-2.5,1)  [acteur][label= $2^5 \cdot 9^3$]{};
 \node (u17) at (-1.5,2)  [acteur][label=$2^6 \cdot 9^3$]{};
 \node (u18) at (-.5,1)  [acteur][label=$2^7 \cdot 9^3$]{};
 \node (u19) at (0,0)   [acteur][label=$2^4 \cdot 9^4$]{};

 \draw (u20) -- (u0);
 \draw (u0) -- (u1);
  \draw (u1) -- (u2);
  \draw (u2) -- (u3);
  \draw (u3) -- (u4);
\draw (u0) -- (u4);

 \draw (u4) -- (u5);
  \draw (u5) -- (u6);
\draw (u6) -- (u7);
  \draw (u7) -- (u8);
  \draw (u8) -- (u9);
  \draw (u9) -- (u5);

 \draw (u9) -- (u10);
 \draw (u10) -- (u11);
  \draw (u11) -- (u12);
  \draw (u12) -- (u13);
  \draw (u13) -- (u14);
 \draw (u10) -- (u14);

 \draw (u14) -- (u15);
  \draw (u15) -- (u16);
  \draw (u16) -- (u17);
  \draw (u17) -- (u18);
 \draw (u18) -- (u19);
 \draw (u15) -- (u19);
\end{tikzpicture}
\caption{$A(C_{5}$ snake Labeling.}
\label{figure99}
\end{figure}
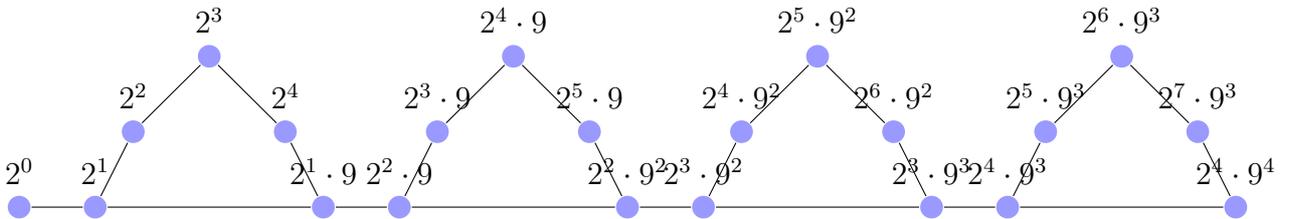

\begin{theorem}
The Olive tree $(T_{k})$ is a difference graph.
\end{theorem}

\begin{proof}
Assume that the Olive tree described as indicated in Fig.\ref{figure75}.
\begin{figure*}[h!]

    \centering

        \begin{tikzpicture}
  [scale=.58,auto=center,every node/.style={circle,fill=blue!20}]

\node(u0) at ( 0,0){Root};

\node(v1) at (-6,-2){$v_1$};
\node(v2) at (-3,-2){$v_2$};
\node(v3) at (0,-2){$v_3$};
\node(v4) at (3,-2){$v_4$};
\node(vn) at (6,-2){$v_n$};

\node(v21) at (-3,-4){$v_{2,1}$};
\node(v31) at (0,-4){$v_{3,1}$};
\node(v41) at (3,-4){$v_{4,1}$};
\node(vn1) at (6,-4){$v_{n,1}$};

\node(v32) at (0,-6){$v_{3,2}$};
\node(v42) at (3,-6){$v_{4,2}$};
\node(vn2) at (6,-6){$v_{n,2}$};

\node(v43) at (3,-8){$v_{4,3}$};
\node(vn3) at (6,-8){$v_{n,3}$};

\node(vn-1) at (6,-11){$v_{n,n-1}$};

\draw(u0) -- (v1);
\draw(u0) -- (v2);
\draw(u0) -- (v3);
\draw(u0) -- (v4);
\draw(u0) -- (vn);

\draw(v2) -- (v21);
\draw(v3) -- (v31);
\draw(v4) -- (v41);
\draw(vn) -- (vn1);

\draw(v32) -- (v31);
\draw(v42) -- (v41);
\draw(vn2) -- (vn1);

\draw(v42) -- (v43);
\draw(vn2) -- (vn3);

\draw [thick, dotted , -] (vn3) -- (vn-1);
\end{tikzpicture}
\caption{Olive tree.}\label{figure75}
    \end{figure*}
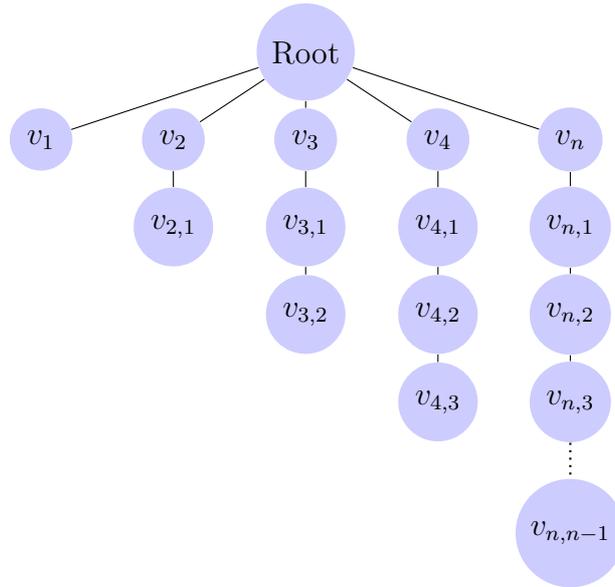%
Assume that the labeling function $f:V(T_{k})\longrightarrow S$ \ is defined as follows:
\begin{eqnarray*}
f(root) &=& 3. \\
f(v_{1}) &=& 6.\\
f(v_{i}) &=& 3+10^{i-1}, \ i=2,3,\ldots ,n.\\
f(v_{i,j}) &=& 2^{j-1} \cdot 10^{i-1},\ j=1,2,\ldots ,i-1.
\end{eqnarray*} \end{proof}

\begin{example}
A difference labeling of the olive tree  $(T_{5})$ labelling is illustrated in Fig.\ref{figure77}.
\end{example}
\begin{figure*}[h!]
    \centering
        \begin{tikzpicture}
  [scale=.58,auto=center,every node/.style={circle,fill=blue!20}]

\node(u0) at ( 0,0){3};

\node(v1) at (-6,-2){6};
\node(v2) at (-3,-2){13};
\node(v3) at (0,-2){103};
\node(v4) at (3,-2){1003};
\node(vn) at (6,-2){10003};

\node(v21) at (-3,-4){10};
\node(v31) at (0,-4){$10^2$};
\node(v41) at (3,-4){$10^3$};
\node(vn1) at (6,-4){$10^4$};

\node(v32) at (0,-6){2 $10^2$};
\node(v42) at (3,-6){2 $10^3$};
\node(vn2) at (6,-6){2 $10^4$};

\node(v43) at (3,-8){4 $10^3$};
\node(vn3) at (6,-8){4 $10^4$};

\node(vn-1) at (6,-11){8 $10^4$};

\draw(u0) -- (v1);
\draw(u0) -- (v2);
\draw(u0) -- (v3);
\draw(u0) -- (v4);
\draw(u0) -- (vn);

\draw(v2) -- (v21);
\draw(v3) -- (v31);
\draw(v4) -- (v41);
\draw(vn) -- (vn1);

\draw(v32) -- (v31);
\draw(v42) -- (v41);
\draw(vn2) -- (vn1);

\draw(v42) -- (v43);
\draw(vn2) -- (vn3);

\draw(vn3) -- (vn-1);
\end{tikzpicture}\
\caption{Olive tree labeling.}\label{figure77}
\end{figure*}
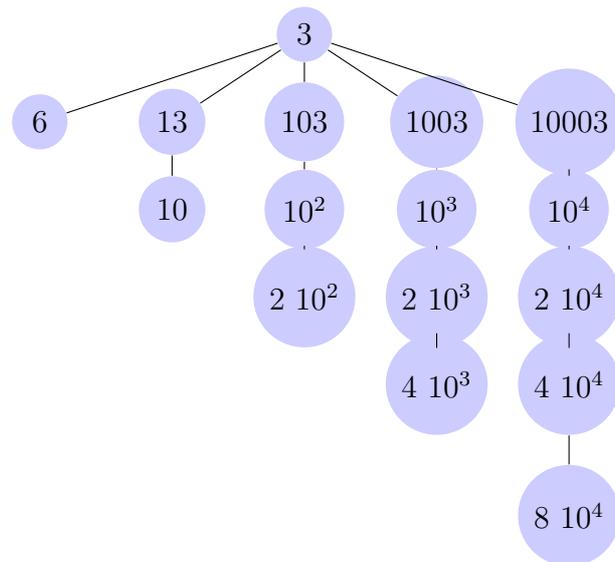

\newpage

\begin{theorem}
The Complete bipartite graph $(K_{2,4})$ is not a difference graph.
\end{theorem}
\begin{proof}
There are many cases which are too involved. The reader can contact the second author for details.
\end{proof}

\section*{Statements and Declarations}
\textbf{Funding}\\
Not applicable\\ \\
\textbf{Conflict of interest}\\
The authors declare that they have no competing interests.

\end{document}